\documentclass[12pt, reqno]{amsart}
\usepackage{amsmath, amssymb, amscd, a4}

\numberwithin{equation}{section}

\newcommand{\mbb}{\mathbb}
\newcommand{\x}{\textbf}

\newcommand{\q}{\quad}

\newcommand{\mrm}{\mathrm}

\theoremstyle{plain}
\newtheorem{thm}{Theorem}[section]
\newtheorem{lem}[thm]{Lemma}

\newtheorem{cor}[thm]{Corollary}
\theoremstyle{definition}

\newtheorem{rmk}[thm]{Remark}

\begin{document}

\title[On a convolution series]{On a convolution series attached to a Siegel Hecke cusp form of degree 2}

\author{Soumya Das} 
\address{School of Mathematics\\
Tata Institute of Fundamental Research\\
Homi Bhabha Road\\
Mumbai -- 400005, India.}
\email{somu@math.tifr.res.in}

\author{Winfried Kohnen} 
\address{Mathematisches Institut\\
Ruprecht-Karls-Universitat Heidelberg \\
D -- 69120 Heidelberg \\
Germany.}
\email{winfried@mathi.uni-heidelberg.de}

\author{Jyoti Sengupta}
\address{School of Mathematics\\
Tata Institute of Fundamental Research\\
Homi Bhabha Road\\
Mumbai -- 400005, India.}
\email{sengupta@math.tifr.res.in}
\date{}

\begin{abstract}
We prove that the \lq\lq naive\rq\rq convolution Dirichlet series $D_2(s)$ attached to a degree $2$ Siegel Hecke cusp form $F$, has a pole at $s=1$. As an application, we write down the asymptotic formula for the partial sums of the squares of the eigenvalues of $F$ with an explicit error term. Further, as a corollary, we are able to show that the abscissa of absolute convergence of the (normalized) spinor zeta function attached to $F$ is $s = 1$. 
\end{abstract}
\maketitle

\section{Introduction}
Let $F$ be a Siegel cusp form of degree $2$ and integral weight $k$ which is an eigenform for all the Hecke operators $T(n)$ with eigenvalues ${\lambda}_F(n)$ (which are necessarily real) and not a Saito--Kurokawa lift. We normalize the eigenvalues by setting $\lambda_n := {\lambda}_F(n)/n^{k-3/2}$ for the rest of the paper. Let us denote the two natural Dirichlet series attached to the sequences $( \lambda_n )_{n \geq 1}$ and $( \lambda^{2}_n )_{n \geq 1}$ by $D_1(s)$ and $D_2(s)$ respectively i.e.,
\[ D_1(s) := \underset{n \geq 1} \sum \lambda_n n^{-s}, \q  D_2(s) := \underset{n \geq 1}\sum  \lambda^{2}_n n^{-s} . \]
Both these series are absolutely convergent for $\mrm{Re}(s) >1$ since we know that $\lambda_F(n) \ll_{\varepsilon} n^{\varepsilon}$ for any $\varepsilon >0$ (see sect.~\ref{prelim} for the notation) as a consequence of the Ramanujan--Petersson conjecture, proved by Weissauer (see \cite{weis}).

We call $D_{2}(s)$ the \lq\lq naive\rq\rq convolution attached to $F$. For an elliptic Hecke cusp form $f$, the \lq\lq naive\rq\rq convolution $D_{2,f}(s)$ attached to $f$ is essentially (up to elementary factors) the Rankin--Selberg $L$-function $L(s, f \otimes f)$. Thus its analytic properties follow from well known properties of $L(s, f \otimes f)$ and its Euler-product is easy to write down. The situation is quite different when the degree is $2$ (and higher). On the one hand, the series $D_1(s)$ is well studied and we know that
\begin{align} \label{defn}
 D_1(s)= \zeta(2s+1)^{-1} Z(s),
\end{align}
where $Z(s):=\tilde{Z}_F(s+k-3/2)$ and $\tilde{Z}_F(s)$ is the spinor-zeta function attached to $F$. The analytic properties of $Z(s)$ are well known (see \cite{and} for example). In fact it admits an analytic continuation to the entire complex plane and also has the expected Euler product and functional equation.

On the other hand, the series $D_2(s)$ turns out {\it not} to be a well-behaved one. We do not know much of it's \lq\lq good\rq\rq analytic properties, which are crucial for obtaining informations about the eigenvalues $\lambda_n$ (say about the growth of $\lambda_n$ as a function of $n$). In particular, nothing is known about the analytic continuation and the location of the possible poles of $D_2(s) $.

The aim of this article is to prove the following theorem. 
\begin{thm} \label{main}
$D_2(s)$ admits an analytic continuation to $\mrm{Re}(s)>1/2$ with the exception of a simple pole at $s=1$ with residue $c_F>0$. 
\end{thm} 
See sect.~\ref{mainproof} for the proof. Our proof relies on the recent result of the functorial transfer of $\mrm{GSp}(4)$ automorphic representations to $\mrm{GL}(4)$ in \cite{saha}. In particular, we use \cite[Theorem~5.2.3]{saha} regarding the analytic properties of the Rankin--Selberg $L$-function attached to the eigenform $F$.

We recall that Weissauer's theorem (cf. above) implies that the abscissa of absolute convergence $\sigma_0$ of $Z(s)$ is less than or equal to $1$. Note that in this connection it was shown previously in \cite{kohnen} that $\sigma_0 \leq 3/2$ and improved to $\sigma_0 \leq 19/18$ in \cite{kohnen3}. Implicit in the proof of Theorem~\ref{main} is the following corollary, see sect.~\ref{mainproof}.

\begin{cor} \label{abs}
The abscissa of absolute convergence of both $D_1(s)$ and $Z(s)$ is $s=1$. 
\end{cor} 

Once we have Theorem~\ref{main}, the Wiener--Ikehara theorem (see e.g. \cite{murty}) immediately implies the following asymptotic formula:
\[ \underset{n \leq x}\sum \lambda_n^{2} = c_F x + o(x), \]
for large $x$, where we have used the standard \lq\lq small-o\rq\rq notation in analytic number theory. However, it is of interest in analytic number theory to get an explicit error term in the above. For example, such an error term gives rise to non-trivial estimates of the eigenvalues. In this regard, we prove the following theorem, see sect.~\ref{estimate}.
\begin{thm} \label{asymptotic}
For large $x$, and any $\eta >0$,
\[  \underset{n \leq x}\sum \lambda^{2}_n = c_F x + O_\eta \big ( k^{5/16} x^{31/32 + \eta} \big) , \]
where the implied constant depends only on $\eta$.
\end{thm}

\begin{rmk}
The constant $31/32$ might not be the best possible in this regard. However, we believe that our method would not give rise to a significantly better bound than that in Theorem~\ref{asymptotic}. 
\end{rmk}
In the context of eigenvalues of cusp forms, it is interesting to know about the existence of nonzero eigenvalues in short intervals. For elliptic cusp forms, this was asked by J. P. Serre in his seminal paper \cite{serre}. As an immediate consequence of Theorem~\ref{asymptotic}, we have the following corollary in this direction. 
\begin{cor}
For every $\eta>0$ there exist a constant $C>0$ depending on $F$ and $\eta$, such that $\lambda_n \neq 0$ for some $n$ in the interval $[x, x + C x^{31/32 + \eta } ]$ for $x$ large.
\end{cor}

As another straightforward application of Theorem~\ref{asymptotic}, we record in Corollary~\ref{away} some non-trivial upper bounds for $\lambda_n$ for $n$ in a set of positive upper natural density. 

\noindent {\bf Acknowledgements.} We would like to thank the School of Mathematics T.I.F.R. Mumbai, where this work was done, for providing excellent working conditions. We also thank the anonymous referee for his suggestions which helped us to correct some points in the paper.

\section{Notation and preliminaries} \label{prelim}
For basic facts about Siegel modular forms we refer to \cite{freitag}. We denote the space of degree $2$ Siegel cusp forms of weight $k$ for the Siegel modular group $\Gamma_2 := \mrm{Sp}_4(\mbb{Z})$ by $S_k(\Gamma_2)$. For $n \in \mbb{N}$, one defines the Hecke operator $T(n)$ on $S_k(\Gamma_2)$ by 
\[ T(n) F = \underset{\gamma \in \Gamma_2 \backslash \Delta_{2,n}} \sum F \mid_k \gamma    \]
where $\Delta_{2,n}$ is the set of integral symplectic similitudes of size $4$ and scale $n$ and 
\[ (F\mid_k \gamma)(Z) := (\det \gamma)^{k/2} \det(CZ+D)^{-k} F(AZ+B)(CZ+D)^{-1}),    \]
for $\gamma = \left( \begin{smallmatrix} A &B \\ C & D \end{smallmatrix} \right)$, $Z$ is an element of the Siegel upper half plane of degree $2$.

The space  $S_k(\Gamma_2)$ has a basis consisting of simultaneous eigenfunctions of all the $T(n)$. The space of Saito--Kurokawa lifts, denoted by $S^{*}_k(\Gamma_2)$, are the lifts of elliptic cusp forms of full level and weight $2k-2$. They can be characterized in terms of the non-entireness of the spinor zeta function attached to the eigenforms. 

Let now $F \in S_k(\Gamma_2)$ be an eigenfunction with $T(n) F = \lambda_F(n) F$ for all $n$. Then it is well known that the $ \lambda_F(n)$ are real and multiplicative: if $(m,n)=1$, then $\lambda_F(mn) = \lambda_F(m)\lambda_F(n)$. To this data, one attaches several $L$-functions, e.g., the standard zeta function and the spinor zeta function $Z(s)$ (normalised as in the Introduction). 

We will adopt the standard notation $s = \sigma + i t$ ($\sigma,t \in \mathbb{R}$) throughout this paper. The function $Z(s)$ admits the following Euler product:
\[ Z(s) = \underset{p}\prod Z_{F,p}(s), \q \text{where } Z_{F,p}(s) = \underset{1\leq i \leq 4}\prod (1 - \beta_{i,p} p^{-s})^{-1}. \] 
Here $\beta_{1,p} := \alpha_{0,p}, \beta_{2,p} := \alpha_{0,p} \alpha_{1,p} , \beta_{3,p} := \alpha_{0,p} \alpha_{2,p} , \beta_{4,p} := \alpha_{0,p} \alpha_{1,p} \alpha_{2,p}$, and the complex numbers $\alpha_{0,p},\alpha_{1,p},\alpha_{2,p}$ are the Satake parameters of $F$. We will drop the suffix $p$ when there is no confusion. By the Ramanujan--Petersson conjecture (now a theorem due to Weissauer),
\[ |\alpha_0| = |\alpha_1| =|\alpha_2|  =1. \]
Using this together with \eqref{defn}, summing the geometric series and observing that $\sigma_0(n) \ll_\varepsilon n^\varepsilon$, it follows that $|\lambda_n|\ll_{\varepsilon} n^{\varepsilon}$ for any $\epsilon >0$. This shows in particular that $Z(s)$ converges absolutely for $\sigma>1$. For the well-known analytic properties of $Z(s)$ we refer the reader to \cite{and}.

For $\sigma>1$, let \[ L(s, F \otimes F) = \underset{p}\prod L_p(s, F \otimes F) \] be the Rankin--Selberg $L$-function attached to $F$; so that we have by definition
\begin{align} \label{Lp}
L_p(X, F \otimes F)^{-1} = \underset{i,j}\prod ( 1 - \beta_i \beta_j X) 
\end{align}
as a polynomial in $X:= p^{-s}$. It has been proved recently in \cite{saha}, that $L(s, F \otimes F)$ has an analytic continuation to the entire complex plane with the exception of two simple poles at $s=0,1$ with positive residue at $s=1$. Furthermore, the \lq\lq completed\rq\rq \ $L$-function satisfies a functional equation and is bounded in vertical strips. We refer the reader to \cite{saha} for details, especially to Theorem~5.2.3 in that paper.  

As is standard in analytic number theory, we have used $a(x) \ll_{\varepsilon} b(x)$ (where $b(x)>0$ for all $x$) to mean $|a(x)| \leq C(\varepsilon) b(x)$ for $x>0$. For a Dirichlet series $D(s):= \underset{n \geq 1}\sum a_n n^{-s}$ convergent in $\sigma \gg 1$, we denote by $\sigma_a(D)$ it's abscissa of absolute convergence, with the convention that $\sigma_a(D) = - \infty$ if $D$ converges absolutely for all $s$.

\section{Pole at $s=1$ of the \lq\lq naive\rq\rq convolution} \label{mainproof}
In this section we wish to prove that $D_2(s)$ has a meromorphic continuation to $\sigma > 1/2$ with a simple pole at $s=1$. However, as mentioned in the Introduction, $D_2(s)$ is not a very well behaved series. Hence we will factorize $D_2$ into two more tractable parts with the hope of gaining something out of it. This is the content of the following lemma, which is essentially adapted from \cite{du-ko} with more details. See also Remark~\ref{r}. 

{\flushleft In the rest of the paper, we abbreviate $L(s):= L(s, F\otimes F)$ and $L_p(s) = L_p(s, F\otimes F)$.}

\begin{lem} \label{H}
For $\sigma>1$ one has the factorization \[ D_2(s) = H(s) L(s, F \otimes F), \] where \[ H(s) = \underset{p}\prod H_p(s) \] is an Euler product, with each $H_p(X)$ a polynomial of degree $\leq 15$ in $X$. It's coefficients are polynomials in the numbers $\beta_i$ and are absolutely bounded. Furthermore, $H(s)$ converges absolutely for $\sigma>1/2$ (so it is holomorphic in this region) and for some absolute constant $A>0$, we have uniformly in $\sigma >1/2$, 
\[ H(s) \ll \sigma^A (\sigma - 1/2)^{-A}. \]
\end{lem} 

\begin{rmk} \label{r}
In \cite[Proposition~2]{du-ko} a result similar to that in the above lemma has been proved for the Fourier coefficients $\lambda_\pi (n)$ of an irreducible, cuspidal representation $\pi$ of $GL(m, \mbb{A}_{\mbb Q})$ where $\mbb{A}_{\mbb Q}$ is the ring of adeles of $\mbb Q$. If we appeal to the result on the transfer of the $GSp(4)$ automorphic representation $\rho_F$ attached to a cuspidal Siegel eigen form $F$ to a \lq\lq nice\rq\rq representation of $GL(4)$ such that $ L(s, \rho_F) := Z(s) = L(s, \pi)$ (\cite{saha}), we get that $\lambda_n = \lambda_\pi (n)$ only for {\it square-free} integers $n$. Thus our lemma is different from that in \cite{du-ko}.  
\end{rmk}

\begin{proof}[Proof of Lemma~\ref{H}]
We look at the $p$-Euler factor of $D_1(s)$ as a rational function in $X$:
\[  D_{1,p}(X) = (1 - X^2 /p) \underset{i}\prod (1 - \beta_i X)^{-1} = (1 - X^2 /p) \left( \underset{i}\sum \frac{c_i}{1 - \beta_i X}    \right), \]
where the last equality is by partial fraction expansion; $c_i$ are rational functions in the $\beta_i$'s. From this, we calculate the coefficient of $X^{\delta}$ from both sides to get
\begin{align*}
\lambda_{p^\delta} &= \underset{i}\sum c_i \beta_i^{\delta} -  \frac{1}{p} \left( \underset{i}\sum c_i \beta_i^{\delta-2} \right) \\
&= \underset{i}\sum r_i \beta_i^{\delta},
\end{align*} 
where $r_i = c_i(1 - \beta_i^{-2}/p)$. Therefore the $p$-Euler factor $D_{2,p}(X)$ can be computed as
\begin{align} \label{Dp}
D_{2,p}(X) = \underset{\delta \geq 0}\sum \lambda^{2}_{p^\delta} X^\delta = \underset{\delta \geq 0}\sum \left( \underset{i,j}\sum  r_i r_j \beta^{\delta}_i \beta^{\delta}_j  \right) X^\delta =\underset{i,j}\sum \frac{r_i r_j}{1 - \beta_i \beta_j X}.
\end{align}
Hence by clearing off the denominator, we get the existence of the Euler factors $H_p(X)$. Clearly the degree in $X$ of $H_p(X)$ is at most $15$. From the equality \[ H_p(X) = \underset{i,j}\prod (1 - \beta_i \beta_j X) \cdot D_{2,p}(X)  \] it is clear that the coefficients of $H_p(X)$ are polynomials in the $\beta_i$. Further since the $\beta_i$ are of absolute value $1$ and $\lambda_{p^\delta}$ are bounded by quantities depending only on $\delta$ (see Remark~\ref{e} for example), we see that for all primes $p$, all the coefficients of $H_p(X)$ are absolutely bounded. 

The reason for the absolute convergence of $H(s)$ in $\sigma>1/2$ is the same as in \cite{du-ko}: the coefficient of $X$ in $H_p(X)$ is zero. Indeed as a polynomial in $X$, differentiating both sides of the equation \[ H_p(X) = D_{2,p}(X) L_p^{-1}(X) \] with respect to $X$ we get  
\begin{align} \label{Hp}
 H'_p(0) = D_{2,p}(0) (L^{-1}_p)'(0) + D'_{2,p}(0) L^{-1}_p(0).  
\end{align}
We note that $D_{2,p}(0) = L_p(0) =1$. From \eqref{Lp}, the coefficient of $X$ in $L^{-1}_p (X)$ is $(L^{-1}_p)'(0) = - \underset{i,j}\sum \beta_i \beta_j = - (\underset{i}\sum \beta _i)^2$. Also by \eqref{Dp} the coefficient of $X$ in $D_{2,p}(X)$ is $D'_{2,p}(0) = \lambda^2_p$. Putting these together, we have by \eqref{Hp} that \[ H'_p(0) = - \big(\underset{i}\sum \beta_i  \big)^2 + \lambda^{2}_p=0.\]
We now turn to the estimate claimed for $H(s)$. In the polynomial defining $H_p(X)$, let $A >0$ be an integer such that the coefficients of $X^2$ are absolutely bounded by $A$ for each prime $p$. Such an $A$ exists by the above. Thus
\begin{align} \label{blabla} | \underset{n \geq 1}\sum h_{p^n} p^{-ns}  |=|H_p(s)| \leq G_p(\sigma), \end{align} where \[ G_p(s) := 1 + A p^{-2s} + |h_{p^3}| p^{-3s} + \cdots |h_{p^{15}}| p^{-15 s}.   \] 
Now let us look at the Dirichlet series 
\[ G_1(\sigma) := \zeta(2 \sigma)^{-A} G(\sigma); \q G(\sigma) = \underset{p} \prod G_p(\sigma) \q (\sigma >1/2). \]
Clearly, the $p$-Euler factor $G_{1,p}(\sigma)$ of $G_1(\sigma)$ satisfies 
\[ G_{1,p}(\sigma) =  (1 - p^{- 2 \sigma})^A G_p(\sigma) = 1 + O(p^{- 3 \sigma}).\]
This shows that the series $G_1(\sigma)$ converges absolutely in $\sigma >1/3$. Hence 
\begin{align} \label{hmm} G(\sigma) \ll \zeta(2 \sigma)^A \ll \sigma^A (\sigma - 1/2)^{-A} \end{align}
for $\sigma>1/2$. Thus the proof of the estimate of $H(s)$ follows from \eqref{blabla} and \eqref{hmm}.
\end{proof}

\begin{proof}[Proof of Theorem~\ref{main}]
From Lemma~\ref{H}, the analytic continuation of $D_2(s)$ follows from the analytic continuation of $L(s)$ to $\mbb{C}$ with simple poles at $s=0,1$ (see \cite[Theorem~5.2.3]{saha}) and the analyticity of $H(s)$ in $\sigma>1/2$. In order to prove that $D_2(s)$ has a simple pole at $s=1$, we claim that it is enough to show that $H(s)$ does not vanish at $s=1$.

Granting the claim for the moment, we show how to complete the proof of Theorem~\ref{main}. First of all, since the pole of $L(s)$ at $s=1$ is simple, the claim shows that the same is true for $D_2(s)$. Thus it follows by Landau's theorem on Dirichlet series with non-negative coefficients that the abscissa of absolute convergence of $D_2(s)$ is $s=1$. Then by the Wiener--Ikehara theorem we obtain for sufficiently large $x$
\begin{align} 
  \underset{n \leq x}\sum \lambda^{2}_n = c_F x + o(x), \label{quantify}
\end{align}
where the constant \[ c_F = \underset{s=1} {\mathrm{Res}} \  D_2(s) = H(1) \cdot \underset{s=1} {\mathrm{Res}} \ L(s) \] is non-zero since $D_2(s)$ has a simple pole at $s=1$ and hence necessarily positive, from \eqref{quantify}. Thus it remains to prove the claim.

{\flushleft \textit{Proof of the claim}:} Suppose to the contrary, i.e., assume that $H(1)=0$. Then $D_2(s)$ has no pole at $s=1$ and since it has non-negative Dirichlet coefficients, by Landau's theorem (since $D_2(s)$ extends to a holomorphic function in a neighbourhood of $s=1$) we get that $D_2(s)$ is convergent for $\sigma > 1 - \varepsilon$ for some $\varepsilon >0$. Then from the above in particular we have for any $\varepsilon'>0$,
\begin{align} \label{hm}
\underset{n \leq x}\sum \lambda_n^{2} \ll_{\varepsilon'} x^{1 - \varepsilon + \varepsilon'}.
\end{align}
This follows from the classical formula for the abscissa of convergence of a Dirichlet series.  

Hence $D_2(1)$ makes sense and is at least $1$. This shows that $H$ has a simple zero at $s=1$. Since $s=1$ is within the region of absolute convergence of $H$, this implies that precisely one Euler factor, say $H_q$, where $q$ is a prime, has a zero at $s=1$ and for primes $p \neq q$, $H_p$ does not vanish at $s=1$.

Let us now remove the $q$-Euler factors from the Dirichlet series under consideration. So we define for $G$, which is one of $D_1,D_2,H$ or $L$,
\[ \tilde{G}(s) := \underset{p \neq q}\prod G_p(s). \]

Now we note that $\tilde{D}_2(s)$ converges absolutely for $\sigma>1$ (since it is a sub-series of $D_2(s)$) and has a pole at $s=1$.  The assertion about the pole follows from the fact that $\tilde{H}$ does not vanish at $s=1$ and that $\tilde{L}$ has a pole at $s=1$, since $L_q$ does not vanish at $s=1$. Indeed, this is immediate as each $q$-Satake parameter has absolute value $1$.

Furthermore $\tilde{D}_2(s)$ has non-negative Dirichlet coefficients. So by Landau's theorem, $\tilde{D}_2$ has abscissa of absolute convergence $\sigma=1$. 

By an argument similar to that in \cite{kohnen} we will prove that the abscissa of absolute convergence of $\tilde{D}_1$ is $s=1$. First of all we easily see that the sum $\underset{(n,q)=1}\sum |\lambda_n|$ diverges. Otherwise the partial sums of the series will be absolutely bounded, which will show that the same is true for the partial sums of $\underset{(n,q)=1}\sum \lambda^{2}_n$ since \[ \underset{(n,q)=1}\sum \lambda^{2}_n \leq \big( \underset{(n,q)=1}\sum | \lambda_n | \big)^2 = O(1).   \]
Hence this will imply that $\tilde{D}_2$ converges absolutely in $\sigma>0$, whereas we have seen above that $\tilde{D}_2$ has it's abscissa of absolute convergence at $s=1$.  

Clearly $\tilde{D}_1$ (being a sub-series of $\tilde{D}_1$) converges absolutely in $\sigma>1$. On the contrary suppose that the abscissa of absolute convergence of $\tilde{D}_1$ is less that $1$. Then we have for some $c>0$
\[ \underset{n \leq x, (n,q)=1}\sum |\lambda_n| \ll x^{1-c}. \]
By Cauchy-Schwartz inequality, we get (as in \cite{kohnen})
\[ \underset{n \leq x, (n,q)=1}\sum \lambda^{2}_n = \underset{n \leq x, (n,q)=1}\sum \lambda^{1/2}_n \times \lambda^{3/2}_n \ll x^{(1-c)/2 + (3 \kappa+1)/2} = x^{1 - c/2 + 3 \kappa/2} \]
for any $\kappa>0$. Choosing $\kappa$ small enough we get a contradiction to the fact that
\[ 1 = \sigma_a(\tilde{D}_2) = \inf \left\{ \alpha \in \mbb{R} \mid  \underset{n \leq x, (n,q)=1}\sum \lambda^{2}_n \ll_\alpha x^\alpha  \right\}.  \]
Note that we can write the above expression for $\sigma_a$ since the series $\underset{n \leq x, (n,q)=1}\sum \lambda^{2}_n$ diverges, since otherwise $\tilde{D}_2(s)$ will converge absolutely for $\sigma>0$. Thus $\sigma_a(\tilde{D}_1)=1$.

We now observe that $D_1(s)$ also has $\sigma_a(D_1) =1$. This is obvious, since
\[ \tilde{D}_1(s) = (1 - q^{-2s-1})^{-1} \underset{i}\prod (1 - \beta_i q^{-s}) \cdot D_1(s); \] 
and for two Dirichlet series $g(s), h(s)$, we have $\sigma_a(gh) \leq \max \{ \sigma_a(g), \sigma_a(h) \}$. The finite product above has $\sigma_a = - 1/2$. So if $\sigma_a(D_1) < 1$, then so will be $\sigma_a(\tilde{D}_1)$. Hence we obtain $\sigma_a(D_1) \geq 1$ and thus $\sigma_a(D_1) = 1$. 

Now we are in a position to complete the proof of the claim. Recall from \eqref{hm} that for any $\varepsilon'>0$, we have $\underset{n \leq x}\sum \lambda_n^{2} \ll_{\varepsilon'} x^{1 - \varepsilon + \varepsilon'}$. This in turn implies by the Cauchy--Schwartz inequality that
\[ \underset{n \leq x}\sum | \lambda_n| \times 1 \ll_{\varepsilon'} x^{(1 - \varepsilon + \varepsilon')/2} \cdot x^{1/2} = x^{ 1-  \varepsilon/2 + \varepsilon'/2}.  \]
Choosing $\varepsilon'$ small enough shows that $D_1(s)$ converges absolutely for $\sigma > 1 - \gamma$ for some $\gamma>0$. This contradicts the fact that $\sigma_a(D_1)=1$. This completes the proof of the claim and hence that of the theorem.
\end{proof}

\begin{proof}[Proof of Corollary~\ref{abs}] The penultimate paragraph of the preceding proof showed that $\sigma_a(D_1)=1$. Now from 
\eqref{defn}, we have $ Z(s) = \zeta(2s+1) D_1(s)$. Hence $\sigma_a(D_1) = \sigma_a(Z)=1$ since $\sigma_a(\zeta^{\pm 1}(2s+1))=0$.  
\end{proof}

\section{The asymptotic result} \label{estimate}
In this section, we show the asymptotic result in Theorem~\ref{asymptotic}. For that, we first prove a \lq\lq convexity\rq\rq estimate for the Rankin--Selberg $L$-function $L(s)$.

\begin{lem} \label{Lest}
Let $2>\sigma >1$. Then for $1-\sigma < \delta < \sigma$ we have,
\begin{align} \label{deltaest}
 L(\delta + it) \ll_{\delta} k^{5(\sigma-\delta)} (\sigma/(\sigma-1))^{16} \ | \delta + it|^{8(\sigma-\delta)} . 
\end{align} 
\end{lem}

\begin{proof}
Let us recall (see \cite[p.79]{saha}) that the completed Rankin--Selberg $L$-function $L^*(s)$, has an analytic continuation to $\mbb{C}$ except simple poles at $s=0,1$. Moreover, it satisfies the following functional equation:
\begin{align} & L^*(s) = L^*(1-s), \text{ where } L^*(s) = L_{\infty}(s)L(s); \label{fnleqn}\\
& L_{\infty}(s) = (\Gamma_{\mbb{R}}(s) \Gamma_{\mbb{R}}(s+1) )^2 ( \Gamma_{\mbb{C}}(s+k-1) \Gamma_{\mbb{C}}(s+k-2))^2  \Gamma_{\mbb{C}}(s+1) \Gamma_{\mbb{C}}(s+2k-3) \nonumber
\end{align}
where \[ \Gamma_{\mbb{R}}(s) = \pi^{-s/2} \Gamma(s/2), \q \Gamma_{\mbb{C}}(s) = 2 (2 \pi)^{-s} \Gamma(s).  \]
Now from \eqref{fnleqn} we have $| L(1-s)| = |L_\infty(s)/ L_\infty(1-s)| |L(s)|$. We next recall the (trivial) bound: 
\[ |L(\sigma + it)| \leq (\sigma/(\sigma-1))^{16} \q (\sigma>1),\]
which is obtained using \eqref{Lp} with the fact that the Satake parameters $\beta_j$ therein are of absolute value $1$; and the trivial estimate $|\zeta(\sigma)| \leq \sigma/(\sigma-1)$ for all $\sigma>1$.  

We then proceed in the usual manner (see \cite{ko-se} for example) by estimating the ratios of Gamma factors to get 
\[ |L_\infty(s)/ L_\infty(1-s)| \ll k^{5(2 \sigma-1)} |1 + it|^{ 8(2 \sigma-1)}. \]
Putting these estimates together, we arrive at 
\begin{align*} 
L(1-\sigma +it) \ll k^{5(2 \sigma-1)} |1 + it|^{8(2 \sigma-1)} (\sigma/(\sigma-1))^{16}.
\end{align*}
The \lq\lq strong convexity\rq\rq principle (see \cite[Lemma~1]{ko-se} for example) then completes the proof. We omit the details.
\end{proof}

\begin{proof}[Proof of Theorem~\ref{asymptotic}]
The proof is by the truncated version of Perron's formula given in \cite{liu}. Let us write for $x,\sigma >1$, $A(x) :=  \underset{n \leq x}\sum \lambda^{2}_n$. 

Let us recall the version of Perron's formula we will use here in the present setup. For $x \geq 2$ and $S,T \geq 2$ we have that
\begin{align*}
A(x) = \frac{1}{2 \pi i} \int_{\sigma - i T}^{\sigma + i T} D_2(s) \frac{x^s}{s} ds + O \Big( \underset{x - x/S \leq n < x + x/S}\sum \lambda^2_n \Big) + O \Big( \frac{x^{\sigma} S D_2(\sigma)}{T} \Big).
\end{align*}

First we note that for $\sigma >1$ (resp. $\delta > 1/2$) by partial summation (resp. by Lemma~\ref{H}), 
\begin{align} \label{d2h}
 D_2(\sigma) \ll \frac{\sigma}{\sigma -1}; \q H(s) \ll \delta^A (\delta - 1/2)^{-A},  \ A>0. 
\end{align} 

We wish to shift the line of integration to the line $\sigma = \delta$, with $1/2<\delta<1$. Using the fact that $H(s)$ converges absolutely in $\sigma>1/2$, by the residue theorem we get, $A(x) = cx + I_1 + I_2(+) - I_2(-) + I_3 +I_4$, where
\begin{align*}
&I_1 =  \frac{1}{2 \pi i} \int_{\delta - i T}^{\delta + i T} D_2(s) \frac{x^s}{s} ds, \q  I_2(\pm) =  \pm \frac{1}{2 \pi i} \underset{\mrm{Im}(s) = \pm T, \delta \leq \mrm{Re}(s) \leq \sigma}\int D_2(s) \frac{x^s}{s} ds \\
& I_3 = O \Big( \underset{x - x/S \leq n < x + x/S}\sum \lambda^2_n \Big), \q I_4 =  O \Big( \frac{x^{\sigma} S D_2(\sigma)}{T} \Big).
\end{align*}

We set $T = x^{\alpha}, S= x^\beta$ ($\alpha>\beta>0$) which will be chosen later. To estimate $I_1$, we use the factorization $D_2(s) = H(s) L(s)$ obtained in Lemma~\ref{H} and the corresponding bounds for $L(s)$ and $H(s)$ from \eqref{deltaest} and \eqref{d2h} respectively. We get,   
\begin{align*}
I_1 &\ll (\delta-1/2)^{-A} k^{5(\sigma-\delta)} (\sigma/(\sigma-1))^{16} \int_{- T}^{ T} \frac{| \delta + it|^{8(\sigma-\delta)}}{|\delta+ it|} dt \cdot x^\delta\\
 & \ll k^{5(\sigma-\delta)} T^{5(\sigma-\delta)} x^\delta \ll k^{5(\sigma-\delta)} x^{8\alpha(\sigma-\delta) + \delta},
\end{align*}
and similarly
\begin{align*}
I_2 \ll (\delta-1/2)^{-A} k^{5(\sigma-\delta)} (\sigma/(\sigma-1))^{16} T^{8(\sigma-\delta)-1} x^\sigma \ll k^{5(\sigma-\delta)} x^{\sigma - \alpha + 8\alpha(\sigma-\delta) }.
\end{align*}

For estimating $I_3$, we fix a small $\eta>0$ and note that $|\lambda^2_n| \ll_\eta n^\eta$ for all $n$. Thus
\begin{align*}
&I_3 \ll (x + x/S)^{1+ \eta} - (x - x/S)^{1+ \eta} = x^{1+ \eta} \big(2(1 + \eta)/S + O(S^{-3}) \big) \ll x^{1+ \eta - \beta},\\
& I_4 \ll \left(\frac{\sigma}{\sigma-1}\right)^{16} x^{\sigma +\beta - \alpha} \ll x^{\sigma +\beta - \alpha}. 
\end{align*}

Now choosing \[ \sigma = 1 + \eta, \ \delta = 15/16 + \eta, \ \alpha = 1/16, \ \beta = 1/32, \q (\eta >0, \text{ small})\]
we see that the exponents of $x$ arising from $I_1, I_2, I_3$ and $I_4$ are all equal to $31/32 + \eta$. This produces the error term to be $O_\eta \big ( k^{5/16} x^{31/32 + \eta} \big)$ and completes the proof of the theorem. 
\end{proof}

\begin{rmk} 
We now make a remark about sequences which are bounded away from zero. Let $(a_n)_{n \geq 1}$ be a sequence of real numbers with the property that for large $x$,
\[ B(x) := \underset{1 \leq n \leq x}\sum a^2_n = c x + o(x), \q c>0.  \]
Then there exist $\alpha >0$ such that $ |a_n| \geq \alpha$, for infinitely many $n$. In other words, the sequence $(a_n)$ is bounded away from $0$. Thus, as an application of Theorem~\ref{asymptotic}, it follows that the sequence $(\lambda_n)_{n \geq 1}$ is bounded away from zero. We note that in fact one can take $\alpha = 1/16$. This is due to \cite[Lemma~5.3]{qu}. 
\end{rmk}




We next collect a few straight forward implications of Theorem~\ref{asymptotic} concerning the upper bounds of the Hecke eigenvalues. For a subset $A$ of the set of the natural numbers, we define the upper natural density of $A$ by:  
\begin{align*}  \bar{\delta}_{\mrm{Nat}}(A) :=  \underset{X \mapsto \infty}{ \lim \sup} \frac{ \# \{ n \leq X \mid n \in A \} } { \# \{ n \leq X \} }.
\end{align*}

\begin{cor} \label{away}
(i) \ There exist infinitely many positive integers $n$ such that \[ |\lambda_n| \leq \sqrt{d(n)},   \]
where $d(n)$ denotes the number of divisors of $n$. \\
(ii) \ On a set of positive upper natural density, we have \[ |\lambda_n| \leq \sqrt{\log n}.  \]
\end{cor}

\begin{proof}
(i) \ Suppose not. Then there exist a positive integer $n_0$ such that for all $n > n_0$, $|\lambda_n| > \sqrt{d(n)}$. This, along with Theorem~\ref{asymptotic} shows that for large $x$
\[ x \log x + O(x) = \underset{n \leq x}\sum d(n) \leq  \underset{n \leq x}\sum \lambda^{2}_n = c_F x + o(x),  \]
a contradiction. 

(ii) \ Suppose not. Then the set $A := \{ n \mid |\lambda_n| \leq \sqrt{\log n}  \}$ has upper natural density $0$. Let us denote $u(x) := \# \{ n \leq x \mid n \in A   \}$. It follows that for large $x$, \[  \underset{x \mapsto \infty}\limsup \ u(x)/x = 0. \]
Since $ \underset{n \leq x}\sum \log n = x \log x + O(x)$, we have, 
\[ \underset{n \leq x, n \not \in A}\sum \log n = x \log x + O(x) + O(u(x) \log x) .  \]
Hence, as in the proof of (i), for large $x$
\[ x \log x + O(x) + O(u(x) \log x) = \underset{n \leq x, n \not \in A}\sum \log n \leq  \underset{n \leq x}\sum \lambda^{2}_n =c_F x + o(x) , \]
a contradiction.
\end{proof}

\begin{rmk} \label{e}
Let us write $Z(s) = \sum b_n n^{-s}$ as a (normalized as before) Dirichlet series, $\sigma>1$. Then, one can easily get an upper bound for the $b_n$'s by comparing $Z(s)$ with $\zeta^4(\sigma)$: \[ |b_n| \leq d_4(n), \q d_4(n) := \underset{ \substack{ {a,b,c,d}  \\ abcd=n} }\sum 1.\]
Hence, from \eqref{defn}, we get an upper bound for the $\lambda_n$'s in a straightforward manner:
\[ |\lambda_n| \leq \underset{ \substack{ m \text{ sq-free}  \\  m^2 \mid n } }\sum  \frac{|\mu(m)|}{m} d_4(n/m^2).  \]
We see that this trivial bound does not imply any of the conclusions of Corollary~\ref{away}. Hence the interest in them.
\end{rmk}


\begin{thebibliography}{11}

\bibitem{and} A. N. Andrianov: {\em Euler products corresponding to Siegel modular forms of genus 2}. Russ. Math. Surv. \x{29}:3, 1974, 45--116.

\bibitem{du-ko} W. Duke, E. Kowalski: {\em A problem of Linnik for elliptic curves and mean-value estimates for automorphic
representations}. Invent. Math. \x{139}, 2000, 1--39.

\bibitem{freitag} E. Freitag: {\em Siegelsche Modulformen}. Grundlehren Math. Wiss. \x{254}, Springer, Berlin, (1983).

\bibitem{iwabook} H. Iwaniec, E. Kowalski: {\em Analytic Number Theory}. A.M.S. Colloquium Publications \x{53}, 2004.

\bibitem{kohnen} W. Kohnen: {\em A note of eigenvalues of Hecke operators on Siegel modular forms of degree two}. Proc. Amer. Math. Soc. \x{113}, no. 3, 1991, 639--642.

\bibitem{kohnen3} W. Kohnen: {\em On the convergence of spinor Zeta functions attached to Hecke eigenforms on $Sp_4(\mathrm{Z})$}. Duke Math. J. \x{76}, no. 3, 1994, 673--681.

\bibitem{kohnen2} W. Kohnen: {\em A short note on Fourier coefficients of Half-integral weight modular forms}. Int. J. Number Theory \x{6}, no. 6, 2010, 1255--1259.

\bibitem{ko-se} W. Kohnen, J. Sengupta: {\em The first negative Hecke eigenvalue of a Siegel cusp form of genus two}. Acta. Arith. \x{129} no. 1, 2007, 53--62.

\bibitem{liu} J. Liu, Y. Ye {\em Perron's Formula and the Prime Number Theorem for Automorphic $L$-functions}. \x{3}, no. 2, (Special Issue: In honor of Leon Simon, Part 1 of 2), 2007, 481--497. 

\bibitem{murty} M. R. Murty, V. K. Murty: {\em Non-vanishing of $L$-functions and applications}. Progress in Mathematics, \x{157}, Birkhäuser Verlag, Basel, 1997, xii+196 pp.

\bibitem{qu} Y. Qu: {\em Linnik-type problems for Automorphic $L$-functions}. J. Number Theory \x{130}, 2010, 786--802.

\bibitem{saha} A. Pitale, A. Saha, R. Schmidt: {\em Transfer of Siegel cusp forms of degree 2}. arXiv:1106.5611.

\bibitem{serre} J. P. Serre: {\em Quelques applications du th\'eor\`eme de densit\'e Chebotarev}. Inst. Hautes Etudes Sci. Publ. Math. \x{54}, 1981, 323--401.

\bibitem{weis} R. Weissauer: {\em Endoscopy for GSp(4) and the Cohomology of Siegel Modular Threefolds}. Springer Lecture notes in Mathematics, \x{1968}, (2009). 


\end{thebibliography}
\end{document}